\documentclass[11pt,reqno]{amsart}

\usepackage{amsmath,amssymb,amsthm,mathtools}
\usepackage{microtype}
\usepackage[colorlinks=true,linkcolor=blue,citecolor=blue,urlcolor=blue]{hyperref}

\newtheorem{theorem}{Theorem}[section]
\newtheorem{lemma}[theorem]{Lemma}

\newtheorem{corollary}[theorem]{Corollary}

\title[On variants of P\'olya's conjecture]{On variants of P\'olya's conjecture}
\author{Songlin Han}
\email{han.songlin.638@s.kyushu-u.ac.jp}
\address{Kyushu University, Fukuoka, Japan}
\date{\today}

\begin{document}

\begin{abstract}
    In this paper, we study a Riesz-type weighted sum of the Liouville function,
    \[
    f(x):=-\sum_{n\le x}\frac{\lambda(n)\log n}{\sqrt n}\log\frac{x}{n}
    \]
    for sufficiently large $x$.
    Motivated by a recent research on the sign criteria for the Riemann Hypothesis arising from weighted prime-counting functions, we investigate the sign behavior of $f(x)$ and its relation to the zeros of the Riemann zeta function. We first prove that if $f(x)$ is non-negative for all sufficiently large $x$, then the Riemann Hypothesis holds. The proof is based on the Mellin transform of $f$ and the analytic properties of $\frac{\zeta(2s)}{\zeta(s)}$.
    Conversely, assuming the Riemann Hypothesis, the Simple Zero Conjecture, and an absolute convergence condition involving the nontrivial zeta zeros, we derive an explicit formula for $f(x)$ in terms of these zeros. In particular, we show that
    \[
    f(x)\sim \frac{(\log x)^3}{12|\zeta(\frac12)|}\qquad (x\to\infty).    
    \]
    Consequently, we show that under these hypotheses, $f(x)$ is eventually positive.
\end{abstract}

\maketitle

\section{Introduction}
Let $n$ be a positive integer and $n=p_1^{a_1} \cdots p_k^{a_k}$ be its prime factorization and $\Omega(n) := a_1+\cdots+a_k$ be the prime divisor function.
Define $\lambda(n) := (-1)^{\Omega(n)}$ be the Liouville function.
P\'olya \cite{Polya} conjectured
\[
  \sum_{n\le m}\lambda(n)\le 0
\]
for all positive $m\ge 2$.

The conjecture has been disproved by Haselgrove \cite{Haselgrove}, and two explicit counterexamples were found by \cite{Lehman, Tanaka}.
Although it is unknown if $\sum_{n\le m}\lambda(n)$ changes sign infinitely often, we still can study the summation of $\lambda$ with different types of weight.
Mossinghoff-Trudgian \cite{MOSSINGHOFF_TRUDGIAN_2012} studied the sum
\[
  L_{\alpha}(x):=\sum_{n\le x}\frac{\lambda(n)}{n^\alpha},\qquad 0\leq \alpha\leq 1.
\]
They investigated the sign behaviour of these sums, its connection with the Riemann hypothesis (RH) and the zeros of the Riemann zeta function $\zeta(s)$,
and identified the case $\alpha=1/2$ as a particularly natural value for persistent sign constancy.
Humphries \cite{Humphries_2013} subsequently studied the distribution behaviour of $L_{\alpha}(x)$.
Under the Riemann hypothesis, the linear independence hypothesis for the zeros of $\zeta(s)$, and a bound on negative discrete moments of $\zeta(s)$,
he proved the existence of limiting logarithmic distributions for suitable normalizations of $L_{\alpha}(x)$ when $0\leq \alpha<1/2$,
showing in particular a negative bias but a positive logarithmic density of positive values.
In the borderline case $\alpha=1/2$,
he conditionally proved that $L_{1/2}(x)<0$ outside a set of logarithmic density zero.
Alkan \cite{ALKAN202190} further developed this by studying the same family of weighted Liouville sums. He showed that, for $1/2<\alpha<1$,
some sign conditions for $L_{\alpha}(x)$ and for its first two derivatives with respect to $\alpha$,
\[
  L_{\alpha}'(x)=-\sum_{n\le x}\frac{\lambda(n)\log n}{n^{\alpha}},\qquad
  L_{\alpha}''(x)=\sum_{n\le x}\frac{\lambda(n)(\log n)^2}{n^{\alpha}},
\]
give criteria equivalent to the Riemann hypothesis.
Recently, Dixit-Chorge \cite{Dixit_Chorge_2026} derived Vorono\"i summation formulas for several arithmetic functions, including the Liouville function $\lambda(n)$,
the M\"obius function $\mu(n)$, and the square of the divisor function.
In the case of $\lambda(n)$,
their formula expresses Liouville sums in terms of a main integral term, a dual oscillatory sum, and contributions from the non-trivial zeros of $\zeta(s)$.
As applications, they obtained Ramanujan--Guinand and Cohen-type identities, as well as conditional oscillation results for the associated Riesz sums.
Their work is therefore concerned primarily with Vorono\"i-type transformations and the oscillatory behaviour of standard Riesz means, rather than sign criteria for the Riemann hypothesis.

\bigskip
\bigskip

The prime number theorem for arithmetic progressions states that the number of primes $3 \pmod 4$ is asymptotically equivalent to the number of primes $1 \pmod 4$.
Chebyshev, however, observed in 1853 that there appear to be more primes congruent to $3$ modulo $4$ than to $1$ modulo $4$, i.e., $\pi(x;4,3)\ge \pi(x;4,1)$ for all $x$.
The phenomenon of Chebyshev's observation is known as the Chebyshev's bias.
Indeed, numerical experiment shows that $\pi(x;4,3)\ge \pi(x;4,1)$ for all $x<26861$.
In spite of that, Leech \cite{MR83001} found that $\pi(x;4,3)-\pi(x;4,1)=-1$ for $x=26861$ as a counterexample.
Rubinstein-Sarnak \cite{RS} showed, under the Generalized Riemann hypothesis and a linear independence hypothesis for the zeros of Dirichlet $L$-functions, that this bias persists in a logarithmic density sense.
In particular, the set of $x$ for which $\pi(x;4,3)>\pi(x;4,1)$ has logarithmic density approximately $0.9959$.
Different types of this problem have been studied by \cite{Fujii, MR272729, MR3925757}.
In particular, Suzuki \cite{Suzuki} recently reformulated the Chebyshev-type bias as a criterion for the Riemann hypothesis in terms of the Riesz mean of the von Mangoldt function
\[
  f_\Lambda(x):=\sum_{n\le x}\frac{\Lambda(n)}{\sqrt n}\log\frac{x}{n}.
\]
The Mellin transform of $-f_\Lambda(x)$ involves
\[
  \frac{1}{(s-\frac{1}{2})^2}\frac{\zeta'(s)}{\zeta(s)},
\]
and the pole of $-\zeta'/\zeta$ at $s=1$ produces the main term $4\sqrt x$ in Suzuki's criterion.
Inspired by this, in this paper, we establish similar theorems for a weighted sum of the Liouville function and RH, by analyzing a different Dirichlet series and its singularity structure.

\bigskip
\bigskip

For $x> 1$, let
\begin{equation}\label{eq:def-f}
  f(x)
  :=-\sum_{n\le x}\frac{\lambda(n)\log n}{\sqrt n}\log\frac{x}{n}
  ,\quad x\ge 1.
\end{equation}
Thus $f$ is the Liouville analogue of Suzuki's $f_\Lambda$ function.
Note that the Dirichlet series of $\lambda$ is
\begin{equation}\label{eq:F-intro}
  F(s):=\sum_{n=1}^\infty\frac{\lambda(n)}{n^s}
  =\prod_p(1+p^{-s})^{-1}
  =\frac{\zeta(2s)}{\zeta(s)},\quad \Re s>1.
\end{equation}
At $s=1$, the factor $1/\zeta(s)$ has a simple zero and $\zeta(2s)$ is regular and non-zero.
Hence $F$ has a simple zero at $s=1$ and $F'$ is regular there.

\bigskip
\bigskip

Using the above notation, the following theorem holds.
\begin{theorem}\label{thm:sign-rh}
  If there exists $x_0\ge 2$ such that
  \[
    f(x)\ge 0,\quad x\ge x_0,
  \]
  then the Riemann hypothesis holds.
\end{theorem}

Let $0<\beta<1$ and $\gamma>0$ be the real and imaginary parts of nontrivial zeros in the upper-half plane $\rho=\beta+i\gamma$ of the Riemann zeta function.
We set the following three conditions.
\begin{description}
  \item[\textup{H1}] Riemann Hypothesis is true, i.e., $\beta=1/2$.
  \item[\textup{H2}] Simple Zero Conjecture is true, i.e., all nontrivial zeros of $\zeta$ are simple.
  \item[\textup{H3}] The following inequality holds.
    \[
      \sum_{\gamma>0}\frac{|\zeta(1+2i\gamma)|}{\gamma^2|\zeta'(\frac{1}{2}+i\gamma)|}<\infty.
    \]
\end{description}
Under these conditions, we give the converse theorem as follows.
\begin{theorem}\label{thm:explicit}
  Assume \textup{H1}-\textup{H3}. Fix $0<\varepsilon<\frac{1}{2}$. Then, as $x\to\infty$,
  \begin{equation}\label{eq:main-asymptotic}
    f(x)=P(\log x)+(\log x)\Phi(\log x)-\Psi(\log x)+O_\varepsilon(x^{-1+\varepsilon}),
  \end{equation}
  where
  \begin{equation}\label{eq:P-def-intro}
    P(L)=-\frac{L^3}{12\zeta(\frac{1}{2})}+\alpha L+\beta
  \end{equation}
  is an explicit cubic polynomial. More precisely, put
  \[
    a_0=\zeta\!\left(\frac{1}{2}\right),\quad
    a_1=\zeta'\!\left(\frac{1}{2}\right),\quad
    a_2=\zeta''\!\left(\frac{1}{2}\right),\quad
    a_3=\zeta'''\!\left(\frac{1}{2}\right),
  \]
  and define the numbers $\gamma_j$ by
  \[
    \zeta(1+z)=\frac1z+\sum_{j=0}^\infty\frac{(-1)^j\gamma_j}{j!}z^j
  \]
  (see \cite[Corollary~1.16]{MV}, with coefficients
  $a_k = (-1)^k\gamma_k / k!$). Then
  \begin{align}\label{eq:alpha-def}
    \alpha&=\frac{a_1^2}{2a_0^3}-\frac{a_2}{4a_0^2}
    -\frac{\gamma_0a_1}{a_0^2}-\frac{2\gamma_1}{a_0},\\[3pt]
    \beta&=-\frac{a_1^3}{a_0^4}+\frac{a_1a_2}{a_0^3}-\frac{a_3}{6a_0^2}
    +\frac{2\gamma_0a_1^2}{a_0^3}-\frac{\gamma_0a_2}{a_0^2}
    +\frac{4\gamma_1a_1}{a_0^2}+\frac{4\gamma_2}{a_0}.\label{eq:beta-def}
  \end{align}
  Here,
  \begin{equation}\label{eq:PhiPsi-def}
    \begin{split}
      \Phi(t)
      &:=2\Re\sum_{\gamma>0}\frac{\zeta(1+2i\gamma)e^{i\gamma t}}{\gamma^2\zeta'(\frac{1}{2}+i\gamma)},\\
      \Psi(t)
      &:=4\Im\sum_{\gamma>0}\frac{\zeta(1+2i\gamma)e^{i\gamma t}}{\gamma^3\zeta'(\frac{1}{2}+i\gamma)}.
    \end{split}
  \end{equation}
\end{theorem}
Notice that the two series converge absolutely and uniformly under \textup{H3}, consequently $\Phi$ and $\Psi$ are bounded continuous functions.

\begin{corollary}\label{cor:positivity}
  Under \textup{H1}-\textup{H3},
  \[
    f(x)\sim \frac{(\log x)^3}{12|\zeta(\frac{1}{2})|}.
  \]
  Hence $f(x)\to +\infty$ as $x$ tends to infinity.
  In particular, there exists $x_0$ such that $f(x)>0$ for all $x\ge x_0$.
\end{corollary}

In the next sections, we first give a quick review on the analytic properties of $\lambda$ and $\zeta$ and some analytic lemmas.
We then prove Theorem~\ref{thm:sign-rh} and Theorem~\ref{thm:explicit} in Section~3.

\bigskip\bigskip

Throughout the paper, we use the Landau and Vinogradov notation.
For a complex valued function $f$, and a positive real function $g$, $f=O(g)$ and $f\ll g$ both mean that $|f|\le C g$ for some constant $C>0$.
Subscripts indicate possible dependence of the implied constant, for example $f\ll_\varepsilon g$ or $f=O_{\varepsilon,\kappa}(g)$.
Unless otherwise stated, implied constants may depend on fixed parameters but not on the main variables such as $x$, $T$, or $n$.
If $f$ is also real, we write $f\sim g$ if $f/g\to1$.
Moreover, if $f$ is also positive, we write $f\asymp g$ if $f\ll g$ and $g\ll f$, $f=o(g)$ if $f/g\to0$.
In this paper, non-trivial zeros of $\zeta(s)$ are denoted by $\rho=\beta+i\gamma$.

\bigskip\bigskip

\section{Preliminaries}

Since $\lambda$ is completely multiplicative and $\lambda(p)=-1$,
\[
  \sum_{k=0}^\infty\frac{\lambda(p^k)}{p^{ks}}
  =\sum_{k=0}^\infty(-p^{-s})^k
  =(1+p^{-s})^{-1}.
\]
For $\Re s>1$, multiplying over primes gives
\[
  F(s)=\prod_p(1+p^{-s})^{-1}
  =\prod_p\frac{1-p^{-s}}{1-p^{-2s}}
  =\frac{\zeta(2s)}{\zeta(s)}.
\]
Differentiation term by term for $\Re s>1$ gives
\begin{equation}\label{eq:Fprime-series}
  F'(s)=-\sum_{n=1}^\infty\frac{\lambda(n)\log n}{n^s}
  ,\quad \Re s>1.
\end{equation}

\begin{lemma}\label{lem:F1}
  The function $F$ is holomorphic at $s=1$ and has a simple zero there. More precisely,
  \begin{equation}\label{eq:F1}
    F(s)=\zeta(2)(s-1)+\bigl(2\zeta'(2)-\gamma_0\zeta(2)\bigr)(s-1)^2+O(|s-1|^3).
  \end{equation}
  In particular, $F'(1)=\zeta(2)$.
\end{lemma}

\begin{proof}
  Let $v=s-1$.
  The Laurent expansion of $\zeta$ at $s=1$ (see \cite[Corollary~1.16]{MV}, with coefficients $a_k = (-1)^k\gamma_k / k!$) gives
  \[
    \zeta(1+v)=\frac1v+\gamma_0-\gamma_1v+O(|v|^2),
  \]
  and hence
  \[
    \frac1{\zeta(1+v)}=v(1-\gamma_0v+O(|v|^2)).
  \]
  Also
  \[
    \zeta(2s)=\zeta(2+2v)=\zeta(2)+2\zeta'(2)v+O(|v|^2).
  \]
  Multiplying the last two expansions gives \eqref{eq:F1}.
\end{proof}

\bigskip
\bigskip

\begin{lemma}\label{lem:mellin}
  For $\Re s>1$,
  \begin{equation}\label{eq:mellin}
    \int_1^\infty f(x)x^{-s+\frac{1}{2}}\frac{dx}{x}
    =\frac{F'(s)}{(s-\frac{1}{2})^2}.
  \end{equation}
\end{lemma}

\begin{proof}
  For $\Re w>0$,
  \[
    \int_1^\infty (\log u)u^{-w}\frac{du}{u}=\frac1{w^2}.
  \]
  For fixed $n\ge 1$, substituting $u=x/n$ gives
  \[
    \frac1{\sqrt n}\int_n^\infty\log\frac{x}{n}\,x^{-s+\frac{1}{2}}\frac{dx}{x}
    =\frac{n^{-s}}{(s-\frac{1}{2})^2},\quad \Re s>\frac{1}{2}.
  \]
  For $\Re s>1$, the series $\sum_n\lambda(n)\log n\,n^{-s}$ is absolutely convergent, so by \eqref{eq:def-f}, Fubini's theorem, and \eqref{eq:Fprime-series},
  \[
    \begin{aligned}
      \int_1^\infty f(x)x^{-s+\frac{1}{2}}\frac{dx}{x}
      &=-\sum_{n=1}^\infty\lambda(n)\log n\cdot\frac{n^{-s}}{(s-\frac{1}{2})^2}  \\
      &=\frac{F'(s)}{(s-\frac{1}{2})^2}.
    \end{aligned}
  \]
\end{proof}

For convenience, we write
\begin{equation}\label{eq:G-def}
  G(s):=\frac{F'(s)}{(s-\frac{1}{2})^2}.
\end{equation}

\bigskip
\bigskip

\begin{lemma}\label{lem:no-real-sing}
  The function $G(s)$ has no real singularity on the real interval $(\frac{1}{2},\infty)$.
\end{lemma}

\begin{proof}
  The factor $(s-\frac{1}{2})^{-2}$ is singular only at $s=\frac{1}{2}$, which is not in the interval
  By Lemma~\ref{lem:F1}, $F'$ is regular at $s=1$.

  It remains to check that $\zeta(s)$ has no real zero in $(\frac{1}{2},\infty)$. This is clear for $s>1$ from the Euler product. For $0<s<1$, we consider
  \[
    (1-2^{1-s})\zeta(s)=\eta(s),\qquad
    \eta(s)=\sum_{n=1}^\infty(-1)^{n-1}n^{-s}.
  \]
  The alternating series $\eta(s)$ is positive for $0<s<1$, while $1-2^{1-s}<0$. Hence $\zeta(s)<0$ on $(0,1)$, in particular on $(\frac{1}{2},1)$. Finally, the pole of $\zeta(2s)$ occurs only at $s=\frac{1}{2}$. Therefore $F'$ and $G$ have no real singularities in $(\frac{1}{2},\infty)$.
\end{proof}

\begin{lemma}[see {\cite[Chapter~II]{Widder}} or {\cite[Proposition~1]{Suzuki}}]\label{lem:LW}
  Let $A(x)$ be non-negative for $x\ge x_0$ with some $x_0\ge 2$ and let
  \[
    \mathcal M(s):=\int_1^\infty A(x)x^{-s}\frac{dx}{x}
  \]
  have finite abscissa of convergence $\sigma_c$.
  Then $s=\sigma_c$ is a singularity of $\mathcal M(s)$.
\end{lemma}

\bigskip\bigskip\bigskip

\section{Proof of the theorems}
\subsection{Proof of Theorem~\ref{thm:sign-rh}}

We now prove Theorem~\ref{thm:sign-rh}. Assume that $f(x)\ge 0$ for all $x\ge x_0$.
We write
\[
  \int_1^\infty f(x)x^{-s+\frac{1}{2}}\frac{dx}{x}=E(s)+I(s),
\]
where
\[
  E(s):=\int_1^{x_0}f(x)x^{-s+\frac{1}{2}}\frac{dx}{x},
  \qquad
  I(s):=\int_{x_0}^\infty f(x)x^{-s+\frac{1}{2}}\frac{dx}{x}.
\]
The function $E(s)$ is entire, because the interval of integration is bounded and for each fixed $x>0$ the factor $x^{-s+\frac{1}{2}}=x^{\frac{1}{2}}e^{-s\log x}$ is entire.
The integral $I(s)$ is the Mellin transform of a non-negative function, and it converges for $\Re s>1$ by Lemma~\ref{lem:mellin}.
Let $\sigma_c$ be its real abscissa of convergence.
Writing $A(x)=f(x)x^{1/2}\mathbf 1_{[x_0,\infty)}(x)$, we have
\[
  I(s)=\int_1^\infty A(x)x^{-s}\frac{dx}{x},
\]
and $A(x)\ge0$.
Hence Lemma~\ref{lem:LW} applies to $I(s)$.

For $\Re s>1$,
\[
  I(s)=G(s)-E(s).
\]
If $\sigma_c$ is finite and $\sigma_c>\frac{1}{2}$, then Lemma~\ref{lem:LW} forces $\sigma_c$ to be a real singularity of $I$.
But $I=G-E$ analytically continues through every real point of $(\frac{1}{2},\infty)$ by Lemma~\ref{lem:no-real-sing}, since $E$ is entire.
Thus this case is impossible.
Therefore $\sigma_c\le \frac{1}{2}$, or else $\sigma_c=-\infty$; in either case, $I$ is analytic on $\Re s>\frac{1}{2}$.

It follows that $G=E+I$ is analytic on $\Re s>\frac{1}{2}$. Since $(s-\frac{1}{2})^{-2}$ is holomorphic and non-zero there, $F'(s)$ is analytic on $\Re s>\frac{1}{2}$.

Suppose that $\zeta$ has a zero $\rho_0$ with $\Re\rho_0>\frac{1}{2}$. Since $\Re(2\rho_0)>1$, we have $\zeta(2\rho_0)\ne 0$.
Hence
\[
  F(s)=\frac{\zeta(2s)}{\zeta(s)}
\]
has a pole at $\rho_0$, of order equal to the multiplicity of the zero of $\zeta$ at $\rho_0$.
Differentiating gives a pole of $F'$ at $\rho_0$, contradicting analyticity on $\Re s>\frac{1}{2}$.
Thus $\zeta$ has no zero in $\Re s>\frac{1}{2}$.
By the symmetry of the nontrivial zeros, all nontrivial zeros lie on $\Re s=\frac{1}{2}$.
This proves RH.

\bigskip\bigskip
\subsection{Proof of Theorem~\ref{thm:explicit}}
In the rest of this section we assume \textup{H1}--\textup{H3}.
Put
\[
  L=\log x,
\]
fix $0<\varepsilon<\frac{1}{2}$, and let
\[
  \sigma_0:=-\frac{1}{2}+\varepsilon.
\]
We also fix $\kappa>1$.
We integrate along the line $\Re s=\kappa$ and then shift it to the left.

We first recall the following formula.
For $\kappa>1$ and $y>0$,
\begin{equation}\label{eq:kernel}
  \frac1{2\pi i}\int_{\kappa-i\infty}^{\kappa+i\infty}\frac{y^u}{u^2}du
  =
  \begin{cases}
    \log y, & y>1,\\
    0, & 0<y\le 1.
  \end{cases}
\end{equation}
This can be easily proved by shifting the contour across the pole at $u=0$.

Let $u=s-\frac{1}{2}$.
By absolute convergence on the line $\Re s=\kappa>1$, we may substitute the Dirichlet series for $F'(s)$ into $G(s)=F'(s)/(s-\frac{1}{2})^2$, which gives
\begin{align*}
  \frac{1}{2\pi i}\int_{\kappa-i\infty}^{\kappa+i\infty}G(s)x^{s-\frac{1}{2}}ds
  &=-\sum_{n=1}^{\infty}\lambda(n)\log n\,
  \frac{1}{2\pi i}\int_{\kappa-i\infty}^{\kappa+i\infty}
  \frac{n^{-s}x^{s-\frac{1}{2}}}{(s-\frac{1}{2})^2}ds  \\
  &=-\sum_{n=1}^{\infty}\frac{\lambda(n)\log n}{\sqrt n}\,
  \frac{1}{2\pi i}\int_{\kappa-\frac{1}{2}-i\infty}^{\kappa-\frac{1}{2}+i\infty}
  \frac{(x/n)^u}{u^2}du  \\
  &=-\sum_{n\le x}\frac{\lambda(n)\log n}{\sqrt n}\log\frac{x}{n}.
\end{align*}
Thus, for $x\ge1$,
\begin{equation}\label{eq:inversion}
  f(x)=\frac1{2\pi i}\int_{\kappa-i\infty}^{\kappa+i\infty}G(s)x^{s-\frac{1}{2}}ds.
\end{equation}

We now shift the line of integration in \eqref{eq:inversion} from $\Re s=\kappa$ to $\Re s=\sigma_0$, using rectangles with horizontal sides at $\Im s=\pm T_n$.
Under \textup{H1} and \textup{H2}, the only poles of $G(s)x^{s-\frac{1}{2}}$ inside such a rectangle are $s=\frac{1}{2}$ and the nontrivial zeros $\rho=\frac{1}{2}+i\gamma$ of $\zeta$ with $|\gamma|<T_n$.
Note that $s=1$ is not a pole. $\zeta$ has a pole at $s=1$, but this makes $F$ vanish rather than blow up.
The trivial zeros of $\zeta(s)$ cause no additional poles of $F(s)=\zeta(2s)/\zeta(s)$, since at $s=-2m$ the zero of the denominator is cancelled by the zero of $\zeta(2s)$.
In any case, these points lie the left of $\Re s=\sigma_0>-\frac12$.

\bigskip
\bigskip

We now compute the residue at $s=\frac{1}{2}$.
Let
\[
  v=s-\frac{1}{2},
\]
and recall that $a_k=\zeta^{(k)}(\frac{1}{2})$ for $0\le k\le3$.
The Laurent expansion of $\zeta$ at $s=1$ gives
\begin{equation}\label{eq:zeta2s-half}
  \zeta(2s)=\zeta(1+2v)=\frac1{2v}+\gamma_0-2\gamma_1v+2\gamma_2v^2+O(|v|^3),
\end{equation}
where $\gamma_j$ are the constants as in Theorem~\ref{thm:explicit}.
On the other hand,
\[
  \zeta\!\left(\frac{1}{2}+v\right)=a_0+a_1v+\frac{a_2}{2}v^2+\frac{a_3}{6}v^3+O(|v|^4),
\]
and since $a_0=\zeta(\frac{1}{2})\ne0$, inverting this series gives
\begin{equation}\label{eq:inv-zeta-half}
  \frac1{\zeta(s)}
  =\frac1{a_0}-\frac{a_1}{a_0^2}v
  +\left(\frac{a_1^2}{a_0^3}-\frac{a_2}{2a_0^2}\right)v^2
  +\left(-\frac{a_1^3}{a_0^4}+\frac{a_1a_2}{a_0^3}-\frac{a_3}{6a_0^2}\right)v^3
  +O(|v|^4).
\end{equation}
We write
\[
  \zeta(2s)=\sum_{j\ge -1}z_jv^j
  \qquad\text{and}\qquad
  \frac1{\zeta(s)}=\sum_{j\ge0}w_jv^j.
\]
Then by \eqref{eq:zeta2s-half} and \eqref{eq:inv-zeta-half}, we get
\[
  z_{-1}=\frac{1}{2},
  \qquad z_0=\gamma_0,
  \qquad z_1=-2\gamma_1,
  \qquad z_2=2\gamma_2,
\]
and
\begin{align*}
  w_0&=\frac1{a_0},\\
  w_1&=-\frac{a_1}{a_0^2},\\
  w_2&=\frac{a_1^2}{a_0^3}-\frac{a_2}{2a_0^2},\\
  w_3&=-\frac{a_1^3}{a_0^4}+\frac{a_1a_2}{a_0^3}-\frac{a_3}{6a_0^2}.
\end{align*}
Thus
\[
  F(s)=\zeta(2s)\frac1{\zeta(s)}
  =\sum_{m\ge -1}\left(\sum_{i+j=m}z_iw_j\right)v^m.
\]
We need the coefficients up to $v^2$.
The coefficient of $v^{-1}$ is
\[
  z_{-1}w_0=\frac1{2a_0},
\]
the constant term is
\[
  B_0=z_{-1}w_1+z_0w_0
  =-\frac{a_1}{2a_0^2}+\frac{\gamma_0}{a_0},
\]
the coefficient of $v$ is
\begin{align*}
  B_1
  &=z_{-1}w_2+z_0w_1+z_1w_0  \\
  &=\frac{a_1^2}{2a_0^3}-\frac{a_2}{4a_0^2}
  -\frac{\gamma_0a_1}{a_0^2}-\frac{2\gamma_1}{a_0},
\end{align*}
and the coefficient of $v^2$ is
\begin{align*}
  B_2
  &=z_{-1}w_3+z_0w_2+z_1w_1+z_2w_0  \\
  &=-\frac{a_1^3}{2a_0^4}+\frac{a_1a_2}{2a_0^3}-\frac{a_3}{12a_0^2}
  +\frac{\gamma_0a_1^2}{a_0^3}-\frac{\gamma_0a_2}{2a_0^2}
  +\frac{2\gamma_1a_1}{a_0^2}+\frac{2\gamma_2}{a_0}.
\end{align*}
Therefore
\begin{equation}\label{eq:F-half-expansion}
  F(s)=\frac{1}{2a_0}v^{-1}+B_0+B_1v+B_2v^2+O(|v|^3),
\end{equation}
and
\begin{equation}\label{eq:Fprime-half-expansion}
  F'(s)=-\frac1{2a_0}v^{-2}+B_1+2B_2v+O(|v|^2).
\end{equation}
Since
\[
  G(s)x^{s-\frac{1}{2}}=\frac{F'(s)}{v^2}e^{vL}
  \qquad\text{and}\qquad
  e^{vL}=1+Lv+\frac{L^2}{2}v^2+\frac{L^3}{6}v^3+O_L(|v|^4),
\]
the coefficient of $v^{-1}$ in $G(s)x^{s-\frac{1}{2}}$ is
\[
  -\frac1{2a_0}\cdot\frac{L^3}{6}+B_1L+2B_2.
\]
Hence
\begin{equation}\label{eq:res-half}
  \operatorname{Res}_{s=\frac{1}{2}}G(s)x^{s-\frac{1}{2}}
  =P(L)=-\frac{L^3}{12\zeta(\frac{1}{2})}+\alpha L+\beta,
\end{equation}
where $\alpha=B_1$ and $\beta=2B_2$.
These are the constants in \eqref{eq:alpha-def} and \eqref{eq:beta-def}.
Since $\zeta(\frac{1}{2})<0$, the coefficient of $L^3$ is positive.

\bigskip\bigskip

We now turn to the residues at the nontrivial zeros.
Let $\rho=\frac{1}{2}+i\gamma$ be a nontrivial zero of $\zeta$, which is simple by \textup{H2}. Near $s=\rho$ we have
\[
  \zeta(s)=\zeta'(\rho)(s-\rho)+O(|s-\rho|^2),
\]
and $\zeta(2\rho)$ is finite since $\rho\ne\frac{1}{2}$. Hence
\[
  F(s)=\frac{\zeta(2\rho)}{\zeta'(\rho)}\frac1{s-\rho}+O(1),
  \qquad
  F'(s)=-\frac{\zeta(2\rho)}{\zeta'(\rho)}\frac1{(s-\rho)^2}+O(1).
\]
Let
\begin{equation}\label{eq:H-def}
  H(s):=\frac{x^{s-\frac{1}{2}}}{(s-\frac{1}{2})^2},
  \qquad
  R_\rho(x):=\operatorname{Res}_{s=\rho}G(s)x^{s-\frac{1}{2}},
\end{equation}
so that $G(s)x^{s-\frac{1}{2}}=F'(s)H(s)$.
Since $F'$ has a pole of order $2$ at $\rho$ with leading coefficient $-\zeta(2\rho)/\zeta'(\rho)$ and $H$ is holomorphic at $\rho$,
\[
  R_\rho(x)=-\frac{\zeta(2\rho)}{\zeta'(\rho)}H'(\rho).
\]
Now
\begin{equation}\label{eq:Hprime}
  H'(s)=x^{s-\frac{1}{2}}
  \left(
    \frac{L}{(s-\frac{1}{2})^2}
    -\frac{2}{(s-\frac{1}{2})^3}
  \right),
\end{equation}
and $\rho-\frac{1}{2}=i\gamma$, so a direct computation gives
\[
  H'(\rho)
  =-\frac{e^{i\gamma L}}{\gamma^2}
  \left(L+\frac{2i}{\gamma}\right).
\]
Since $2\rho=1+2i\gamma$, we obtain
\begin{equation}\label{eq:Rrho}
  R_\rho(x)
  =\frac{\zeta(1+2i\gamma)e^{i\gamma L}}
  {\gamma^2\zeta'(\frac{1}{2}+i\gamma)}
  \left(L+\frac{2i}{\gamma}\right).
\end{equation}

We next pair conjugate zeros.
For each $\gamma>0$, put
\[
  A_\gamma(L):=
  \frac{\zeta(1+2i\gamma)e^{i\gamma L}}
  {\gamma^2\zeta'(\frac{1}{2}+i\gamma)},
\]
so that \eqref{eq:Rrho} reads $R_{\frac{1}{2}+i\gamma}(x)=A_\gamma(L)\bigl(L+\frac{2i}{\gamma}\bigr)$.
Since $\zeta(\overline{s})=\overline{\zeta(s)}$ and $L$ is real, we have $R_{\frac{1}{2}-i\gamma}(x)=\overline{R_{\frac{1}{2}+i\gamma}(x)}$, and therefore
\[
  R_{\frac{1}{2}+i\gamma}(x)+R_{\frac{1}{2}-i\gamma}(x)
  =2\Re\left(A_\gamma(L)\left(L+\frac{2i}{\gamma}\right)\right)
  =2L\Re (A_\gamma(L))-\frac4\gamma\Im (A_\gamma(L)).
\]
Summing over $\gamma>0$ and recalling the definitions \eqref{eq:PhiPsi-def} of $\Phi$ and $\Psi$, we get
\begin{equation}\label{eq:zero-pairing}
  \sum_\rho R_\rho(x)=L\Phi(L)-\Psi(L).
\end{equation}
By \textup{H3}, the series defining $\Phi$ converges absolutely and uniformly in $L$, and so does the series defining $\Psi$, which has an extra factor $1/\gamma$.
Hence $\Phi$ and $\Psi$ are bounded.

\bigskip
\bigskip

It remains to estimate the integrals along the horizontal and vertical lines of the rectangle. 
We will use the functional equation
\[
  \zeta(s)=\chi(s)\zeta(1-s),
\]
where
\begin{equation}\label{eq:chi-def}
  \chi(s):=\frac{2^{s-1}\pi^s\sec(\pi s/2)}{\Gamma(s)}
  =2^s\pi^{s-1}\sin\frac{\pi s}{2}\,\Gamma(1-s)
  =\pi^{s-1/2}\frac{\Gamma((1-s)/2)}{\Gamma(s/2)}.
\end{equation}
By Stirling's formula, uniformly in any fixed vertical strip,
\begin{equation}\label{eq:chi-size}
  |\chi(\sigma+it)|\asymp |t|^{1/2-\sigma}
  ,\quad |t|\ge2.
\end{equation}
See \cite[Sections~2.6, 4.12, and 5.1]{Titchmarsh}.

\begin{lemma}[see {\cite[Theorem~14.16]{Titchmarsh}}]\label{lem:quoted-horizontal}
  Assume RH. There exists an absolute constant $C_0>0$ such that, for all sufficiently large $T$, the interval $[T,T+1]$ contains a number $t$ for which
  \begin{equation}\label{eq:quoted-horizontal-bound}
    |\zeta(\sigma+it)|^{-1}
    \le \exp\!\left(C_0\frac{\log t}{\log\log t}\right),
    \qquad \frac{1}{2}\le \sigma\le 2.
  \end{equation}
\end{lemma}

\begin{lemma}\label{lem:H4-from-RH}
  Assume RH. Fix $\kappa>1$ and $0<\varepsilon<\frac{1}{2}$, and put $\sigma_0=-\frac{1}{2}+\varepsilon$. Then there exists a sequence $T_n\to\infty$ such that
  \begin{equation}\label{eq:H4-proved}
    \frac1{T_n^2}
    \int_{\sigma_0}^{\kappa}
    \left|\frac{\zeta(2\sigma+2iT_n)}{\zeta(\sigma+iT_n)}\right|d\sigma
    \longrightarrow0.
  \end{equation}
  The same statement holds with $-iT_n$ in place of $+iT_n$. More precisely, for every $\eta>0$,
  \begin{equation}\label{eq:H4-quant}
    \int_{\sigma_0}^{\kappa}
    \left|\frac{\zeta(2\sigma+2iT_n)}{\zeta(\sigma+iT_n)}\right|d\sigma
    \ll_{\varepsilon,\kappa,\eta}T_n^{\frac{1}{2}-\varepsilon+\eta+o(1)}.
  \end{equation}
\end{lemma}

\begin{proof}
  Choose $T_n\in[n,n+1]$ as in Lemma~\ref{lem:quoted-horizontal}, and put
  \[
    \mathcal B(T_n):=
    \exp\!\left(C_0\frac{\log T_n}{\log\log T_n}\right)
    =T_n^{C_0/\log\log T_n}
    =T_n^{o(1)}.
  \]
  By Lemma~\ref{lem:quoted-horizontal},
  \[
    |\zeta(u+iT_n)|^{-1}\ll \mathcal B(T_n),
    \qquad \frac{1}{2}\le u\le2.
  \]
  If $\kappa>2$, then for $2\le u\le\kappa$, $1/\zeta(u+iT_n)$ is bounded, since the Dirichlet series of $1/\zeta$ converges absolutely there. Hence
  \begin{equation}\label{eq:den-right}
    |\zeta(u+iT_n)|^{-1}\ll_\kappa \mathcal B(T_n),
    \qquad \frac{1}{2}\le u\le\max(2,\kappa).
  \end{equation}

  Now let $\sigma_0\le\sigma\le\frac{1}{2}$. Since
  \[
    1-\sigma\in\left[\frac{1}{2},\frac32-\varepsilon\right]\subset\left[\frac{1}{2},2\right],
  \]
  applying \eqref{eq:den-right} to $1-\sigma$ and taking complex conjugates, we get
  \[
    |\zeta(1-\sigma-iT_n)|^{-1}\ll \mathcal B(T_n).
  \]
  Together with the functional equation and \eqref{eq:chi-size}, this gives
  \begin{equation}\label{eq:den-left}
    |\zeta(\sigma+iT_n)|^{-1}
    =|\chi(\sigma+iT_n)|^{-1}|\zeta(1-\sigma-iT_n)|^{-1}
    \ll_{\varepsilon}T_n^{\sigma-\frac{1}{2}}\mathcal B(T_n)
    \qquad \left(\sigma_0\le\sigma\le\frac{1}{2}\right).
  \end{equation}

  We also need upper bounds for the numerator $|\zeta(2\sigma+2iT_n)|$.
  Let
  \[
    w=2\sigma+2iT_n=r+i\tau,
    \qquad
    r=2\sigma,
    \qquad
    \tau=2T_n.
  \]
  We split the range into three parts $r\le0$, $0\le r\le1$, and $r\ge1$, that is, $\sigma\le0$, $0\le\sigma\le\frac{1}{2}$, and $\sigma\ge\frac{1}{2}$.

  First let $\sigma_0\le\sigma\le0$, so that $r\le0$. By the functional equation,
  \[
    \zeta(r+i\tau)=\chi(r+i\tau)\zeta(1-r-i\tau),
  \]
  and $1-r=1-2\sigma\in[1,2-2\varepsilon]$. Hence, uniformly in this range (see \cite[Theorem~3.5]{Titchmarsh}),
  \[
    |\zeta(1-r-i\tau)|\ll\log|\tau|\ll_\eta |\tau|^\eta,
  \]
  and \eqref{eq:chi-size} gives
  \[
    |\zeta(r+i\tau)|
    \ll_{\varepsilon,\eta}|\tau|^{1/2-r+\eta}.
  \]
  Since $r=2\sigma$ and $\tau=2T_n$, this is
  \begin{equation}\label{eq:numerator-left}
    |\zeta(2\sigma+2iT_n)|
    \ll_{\varepsilon,\eta}T_n^{\frac{1}{2}-2\sigma+\eta}
    ,\quad \sigma_0\le\sigma\le0.
  \end{equation}

  Next let $0\le\sigma\le\frac{1}{2}$, so that $0\le r\le1$. The bound
  \[
    \zeta(r+i\tau)\ll_\eta |\tau|^{(1-r)/2+\eta}
    ,\quad 0\le r\le1,\ |\tau|\ge2
  \]
  (see \cite[Section~5.1, (5.1.4)]{Titchmarsh}) gives
  \begin{equation}\label{eq:numerator-strip}
    |\zeta(2\sigma+2iT_n)|
    \ll_\eta T_n^{(1-2\sigma)/2+\eta}
    =T_n^{\frac{1}{2}-\sigma+\eta}
    \qquad \left(0\le\sigma\le\frac{1}{2}\right).
  \end{equation}

  Finally let $\frac{1}{2}\le\sigma\le\kappa$, so that $1\le r\le2\kappa$. Here $\zeta(r+i\tau)\ll_\kappa\log|\tau|$ uniformly (note that everything converges absolutely for $r\ge2$, see \cite[Theorem~3.5]{Titchmarsh} for the rest), and since $\log T_n\ll_\eta T_n^\eta$,
  \begin{equation}\label{eq:numerator-right}
    |\zeta(2\sigma+2iT_n)|
    \ll_{\kappa,\eta}T_n^\eta
    \qquad \left(\frac{1}{2}\le\sigma\le\kappa\right).
  \end{equation}
  Combining \eqref{eq:numerator-left}, \eqref{eq:numerator-strip}, and \eqref{eq:numerator-right}, we have for every fixed $\eta>0$,
  \begin{equation}\label{eq:numerator-bounds}
    |\zeta(2\sigma+2iT_n)|
    \ll_{\varepsilon,\kappa,\eta}
    \begin{cases}
      T_n^{\frac{1}{2}-2\sigma+\eta}, & \sigma_0\le\sigma\le0,\\[2pt]
      T_n^{\frac{1}{2}-\sigma+\eta}, & 0\le\sigma\le\frac{1}{2},\\[2pt]
      T_n^\eta, & \frac{1}{2}\le\sigma\le\kappa.
    \end{cases}
  \end{equation}

  Multiplying the bounds for the numerator and the denominator, we get for $\sigma_0\le\sigma\le0$, by \eqref{eq:numerator-bounds} and \eqref{eq:den-left},
  \[
    \left|\frac{\zeta(2\sigma+2iT_n)}{\zeta(\sigma+iT_n)}\right|
    \ll T_n^{\frac{1}{2}-2\sigma+\eta}\cdot T_n^{\sigma-\frac{1}{2}}\mathcal B(T_n)
    =T_n^{-\sigma+\eta}\mathcal B(T_n);
  \]
  for $0\le\sigma\le\frac{1}{2}$, by the same denominator bound,
  \[
    \left|\frac{\zeta(2\sigma+2iT_n)}{\zeta(\sigma+iT_n)}\right|
    \ll T_n^{\frac{1}{2}-\sigma+\eta}\cdot T_n^{\sigma-\frac{1}{2}}\mathcal B(T_n)
    =T_n^\eta\mathcal B(T_n);
  \]
  and for $\frac{1}{2}\le\sigma\le\kappa$, by \eqref{eq:den-right} and \eqref{eq:numerator-bounds}, again
  \[
    \left|\frac{\zeta(2\sigma+2iT_n)}{\zeta(\sigma+iT_n)}\right|
    \ll T_n^\eta\mathcal B(T_n).
  \]
  Thus
  \begin{equation}\label{eq:quotient-bounds}
    \left|\frac{\zeta(2\sigma+2iT_n)}{\zeta(\sigma+iT_n)}\right|
    \ll_{\varepsilon,\kappa,\eta}
    \begin{cases}
      T_n^{-\sigma+\eta}\mathcal B(T_n), & \sigma_0\le\sigma\le0,\\[2pt]
      T_n^\eta\mathcal B(T_n), & 0\le\sigma\le\kappa.
    \end{cases}
  \end{equation}

  Integrating \eqref{eq:quotient-bounds}, the interval $[\sigma_0,0]$ contributes
  \begin{align*}
    \int_{\sigma_0}^{0}T_n^{-\sigma+\eta}\mathcal B(T_n)d\sigma
    &=T_n^\eta\mathcal B(T_n)\int_{\sigma_0}^{0}T_n^{-\sigma}d\sigma  \\
    &\ll_\varepsilon T_n^\eta\mathcal B(T_n)T_n^{-\sigma_0}
    =T_n^{\frac{1}{2}-\varepsilon+\eta}\mathcal B(T_n),
  \end{align*}
  since $-\sigma_0=\frac{1}{2}-\varepsilon$, while the interval $[0,\kappa]$ has bounded length and contributes
  \[
    \ll_{\kappa,\eta}T_n^\eta\mathcal B(T_n),
  \]
  which is smaller. Hence
  \[
    \int_{\sigma_0}^{\kappa}
    \left|\frac{\zeta(2\sigma+2iT_n)}{\zeta(\sigma+iT_n)}\right|d\sigma
    \ll_{\varepsilon,\kappa,\eta}T_n^{\frac{1}{2}-\varepsilon+\eta}\mathcal B(T_n)
    =T_n^{\frac{1}{2}-\varepsilon+\eta+o(1)}.
  \]
  This proves \eqref{eq:H4-quant}, and dividing by $T_n^2$ gives \eqref{eq:H4-proved}. The estimate with $-iT_n$ follows from $\zeta(\overline{s})=\overline{\zeta(s)}$.
\end{proof}

\bigskip\bigskip

\begin{lemma}\label{lem:vertical}
  On the line $\Re s=\sigma_0=-\frac{1}{2}+\varepsilon$,
  \begin{equation}\label{eq:vertical-bound}
    G(s)x^{s-\frac{1}{2}}
    \ll_\varepsilon x^{-1+\varepsilon}|t|^{-\frac32-\varepsilon}\log(|t|+3)
    ,\quad s=\sigma_0+it,\ |t|\ge2.
  \end{equation}
  Consequently,
  \begin{equation}\label{eq:vertical-integral}
    \int_{\sigma_0-i\infty}^{\sigma_0+i\infty}G(s)x^{s-\frac{1}{2}}ds
    =O_\varepsilon(x^{-1+\varepsilon}).
  \end{equation}
\end{lemma}

\begin{proof}
  Let $s=\sigma_0+it$ with $|t|\ge2$.
  Since
  \[
    1-s=\frac32-\varepsilon-it
    \qquad\text{and}\qquad
    1-2s=2-2\varepsilon-2it
  \]
  lie in the half-plane $\Re z\ge\frac32-\varepsilon>1$, the functions $\zeta$, $1/\zeta$, and $\zeta'/\zeta$ are bounded at these points, with constants depending only on $\varepsilon$. By the functional equation and \eqref{eq:chi-size},
  \[
    |\zeta(s)|=|\chi(s)|\,|\zeta(1-s)|\asymp_\varepsilon |t|^{1-\varepsilon},
    \qquad
    |\zeta(2s)|=|\chi(2s)|\,|\zeta(1-2s)|
    \ll_\varepsilon |t|^{\frac{1}{2}-2\sigma_0}
    =|t|^{\frac32-2\varepsilon},
  \]
  and hence
  \begin{equation}\label{eq:F-left-bound}
    F(s)\ll_\varepsilon |t|^{\frac{1}{2}-\varepsilon}.
  \end{equation}

  We also need a bound for $F'(s)$.
  By the logarithmic derivative of the functional equation, we have
  \[
    \frac{\zeta'}{\zeta}(z)
    =\frac{\chi'}{\chi}(z)-\frac{\zeta'}{\zeta}(1-z).
  \]
  By Stirling's formula, $\chi'(z)/\chi(z)\ll\log(|\Im z|+3)$ in any fixed vertical strip, and $\zeta'/\zeta(1-z)$ is bounded when $\Re(1-z)\ge\frac32-\varepsilon$. Applying this with $z=s$ and $z=2s$ gives
  \[
    \frac{\zeta'}{\zeta}(s)\ll_\varepsilon\log(|t|+3),
    \qquad
    \frac{\zeta'}{\zeta}(2s)\ll_\varepsilon\log(|t|+3).
  \]
  Since
  \[
    \frac{F'}{F}(s)
    =2\frac{\zeta'}{\zeta}(2s)-\frac{\zeta'}{\zeta}(s),
  \]
  \eqref{eq:F-left-bound} gives
  \begin{equation}\label{eq:Fprime-left-bound}
    F'(s)\ll_\varepsilon |t|^{\frac{1}{2}-\varepsilon}\log(|t|+3).
  \end{equation}
  Finally, $|x^{s-\frac{1}{2}}|=x^{\sigma_0-\frac{1}{2}}=x^{-1+\varepsilon}$ and $|s-\frac{1}{2}|^{-2}\ll |t|^{-2}$.
  Together with $G(s)=F'(s)/(s-\frac{1}{2})^2$, this proves \eqref{eq:vertical-bound}.

  The right-hand side of \eqref{eq:vertical-bound} is integrable over $|t|\ge2$, and on the segment $|t|\le2$ the integrand is $O_\varepsilon(x^{-1+\varepsilon})$. Hence \eqref{eq:vertical-integral} follows.
\end{proof}

\bigskip\bigskip

\begin{lemma}\label{lem:horizontal}
  Let $T_n$ be the sequence from Lemma~\ref{lem:H4-from-RH}. Then, for each fixed $x>1$, the integrals of $G(s)x^{s-\frac{1}{2}}$ over the horizontal segments from $\sigma_0\pm iT_n$ to $\kappa\pm iT_n$ tend to $0$ as $n\to\infty$.
\end{lemma}

\begin{proof}
  Recall $H(s)$ from \eqref{eq:H-def}, so that $G(s)x^{s-\frac{1}{2}}=H(s)F'(s)$.
  On the upper line segment $s=\sigma+iT_n$, integration by parts gives
  \begin{equation}\label{eq:IBP-horizontal}
    \int_{\sigma_0}^{\kappa}H(\sigma+iT_n)F'(\sigma+iT_n)d\sigma
    =[H(s)F(s)]_{\sigma_0+iT_n}^{\kappa+iT_n}
    -\int_{\sigma_0}^{\kappa}H'(\sigma+iT_n)F(\sigma+iT_n)d\sigma.
  \end{equation}
  For $\sigma\in[\sigma_0,\kappa]$ we have $|\sigma+iT_n-\frac{1}{2}|\asymp T_n$. Since $x$ is fixed,
  \begin{equation}\label{eq:H-Hprime-horizontal}
    |H(\sigma+iT_n)|\ll_{x,\varepsilon,\kappa}T_n^{-2},
  \end{equation}
  and by \eqref{eq:Hprime},
  \begin{equation}\label{eq:Hprime-horizontal}
    |H'(\sigma+iT_n)|\ll_{x,\varepsilon,\kappa}(1+L)T_n^{-2}.
  \end{equation}

  We first handle the endpoint terms in \eqref{eq:IBP-horizontal}. At the right endpoint $\Re s=\kappa>1$, both $\zeta(2s)$ and $1/\zeta(s)$ are bounded, so $F(\kappa+iT_n)\ll_\kappa1$ and
  \[
    H(\kappa+iT_n)F(\kappa+iT_n)
    \ll_{x,\kappa}T_n^{-2}=o(1).
  \]
  At the left endpoint $s=\sigma_0+iT_n$, \eqref{eq:F-left-bound} gives $F(\sigma_0+iT_n)\ll_\varepsilon T_n^{\frac{1}{2}-\varepsilon}$, and $|H(\sigma_0+iT_n)|\ll_{x,\varepsilon}x^{-1+\varepsilon}T_n^{-2}$, so
  \[
    H(\sigma_0+iT_n)F(\sigma_0+iT_n)
    \ll_{x,\varepsilon}x^{-1+\varepsilon}T_n^{-\frac32-\varepsilon}=o(1).
  \]
  Thus the boundary term in \eqref{eq:IBP-horizontal} is $o(1)$.

  For the integral term, \eqref{eq:Hprime-horizontal} and Lemma~\ref{lem:H4-from-RH} give
  \begin{align*}
    \int_{\sigma_0}^{\kappa}|H'(\sigma+iT_n)F(\sigma+iT_n)|d\sigma
    &\ll_{x,\varepsilon,\kappa}(1+L)T_n^{-2}
    \int_{\sigma_0}^{\kappa}|F(\sigma+iT_n)|d\sigma  \\
    &=(1+L)\,o(1),
  \end{align*}
  which tends to $0$ since $x$, and hence $L$, is fixed. This proves the claim for the upper segment; the lower segment is treated in the same way, using $\zeta(\overline{s})=\overline{\zeta(s)}$.
\end{proof}

\bigskip\bigskip

\begin{proof}[Proof of Theorem~\ref{thm:explicit}]
  Let $\mathcal R_n$ be the rectangle with vertical sides $\Re s=\kappa$ and $\Re s=\sigma_0$, and horizontal sides $\Im s=\pm T_n$, where $T_n$ is the sequence from Lemma~\ref{lem:H4-from-RH}. By the residue theorem,
  \begin{equation}\label{eq:truncated-contour}
    \frac1{2\pi i}\int_{\kappa-iT_n}^{\kappa+iT_n}G(s)x^{s-\frac{1}{2}}ds
    =P(L)+\sum_{\substack{\rho:\ |\Im\rho|<T_n}}R_\rho(x)
    +\frac1{2\pi i}\int_{\sigma_0-iT_n}^{\sigma_0+iT_n}G(s)x^{s-\frac{1}{2}}ds+o(1),
  \end{equation}
  where the $o(1)$ term, as $n\to\infty$ with $x$ fixed, is the contribution of the two horizontal sides, by Lemma~\ref{lem:horizontal}.

  Let $n\to\infty$. The left-hand side of \eqref{eq:truncated-contour} tends to the integral in \eqref{eq:inversion}, by absolute convergence on $\Re s=\kappa$. The vertical integral is $O_\varepsilon(x^{-1+\varepsilon})$ by Lemma~\ref{lem:vertical}, and the sum over the zeros converges absolutely by \textup{H3} and \eqref{eq:Rrho}. Therefore
  \[
    f(x)=P(L)+\sum_\rho R_\rho(x)+O_\varepsilon(x^{-1+\varepsilon}),
  \]
  and by \eqref{eq:zero-pairing} this is
  \[
    f(x)=P(L)+L\Phi(L)-\Psi(L)+O_\varepsilon(x^{-1+\varepsilon}),
  \]
  which proves Theorem~\ref{thm:explicit}.
\end{proof}

\bigskip\bigskip

\subsection{Proof of the Corollary}

By Theorem~\ref{thm:explicit},
\[
  f(x)=P(L)+L\Phi(L)-\Psi(L)+O_\varepsilon(x^{-1+\varepsilon}),
  \qquad L=\log x.
\]
Since $\zeta(\frac{1}{2})<0$,
\[
  P(L)=\frac{L^3}{12|\zeta(\frac{1}{2})|}+\alpha L+\beta.
\]
By \textup{H3}, the functions $\Phi$ and $\Psi$ are bounded. Hence there exist constants $M_\Phi,M_\Psi<\infty$ such that
\[
  |\Phi(t)|\le M_\Phi,
  \qquad
  |\Psi(t)|\le M_\Psi
  ,\quad t\in\mathbb R.
\]
It follows that
\[
  f(x)=\frac{L^3}{12|\zeta(\frac{1}{2})|}+O(L)+O_\varepsilon(x^{-1+\varepsilon}).
\]
Since $x^{-1+\varepsilon}=o(1)$ and $L=o(L^3)$ as $x\to\infty$, we obtain
\[
  f(x)\sim \frac{(\log x)^3}{12|\zeta(\frac{1}{2})|}.
\]
In particular, $f(x)\to+\infty$, and so $f(x)>0$ for all sufficiently large $x$. This proves Corollary~\ref{cor:positivity}.

\bigskip
\section*{Acknowledgements}
The author would like to thank his supervisor, Professor Ade Irma Suriajaya, for her kind advice.

\bigskip
\nocite{*}
\bibliographystyle{amsplain}
\bibliography{references}
\end{document}